\documentclass[12pt,reqno]{amsart}

\usepackage{etex}
\usepackage[colorlinks=true,linkcolor=blue,citecolor=blue,urlcolor=red]{hyperref}
\usepackage{array}
\usepackage{bbding}
\usepackage{graphicx}
\usepackage{amssymb,latexsym,amsmath}
\usepackage{pstricks-add}
\usepackage{fancynum}\setfnumgsym{\,\,}
\usepackage{bm}
\usepackage{xy}
\usepackage{pdfpages} 
\xyoption{all}

\theoremstyle{theorem}
\newtheorem{theorem}{Theorem}[section]

\newtheorem{proposition}[theorem]{Proposition}

\numberwithin{equation}{section}
\theoremstyle{definition}
\newtheorem{remark}[theorem]{Remark}

\newtheorem{example}[theorem]{Example}

\newcommand{\cp}{\mathbb{C}P}
\newcommand{\zz}{\mathbb{Z}}

\begin{document}

\title{Remarks on $5$-dimensional complete intersections}

\author{Jianbo Wang}
\address{Department of Mathematics, School of Science, Tianjin University\newline
\indent Weijin Road 92, Nankai District, Tianjin  300072, P.R.China}
\email{wjianbo@tju.edu.cn}
\thanks{The author is supported by NSFC grant No.11001195 and Beiyang Elite Scholar Program of Tianjin University.}

\maketitle

\begin{abstract}
This paper will give some examples of diffeomorphic complex $5$-dimensional complete intersections and remarks on these examples. Then a result on the existence of diffeomorphic complete intersections that belong to components of the moduli space of different dimensions will be given as a supplement to the results of P.Br{\"u}ckmann (J. reine angew. Math. \textbf{476} (1996), 209--215; \textbf{525} (2000), 213--217).
\end{abstract}

\section{Introduction}

Let $X_n(\underline{d})\subset \cp^{n+r}$ be a smooth complete
intersection of  multidegree $\underline{d}:=(d_1,\cdots,d_r)$, i.e,
the transversal intersections of hypersurfaces of degrees $d_1,
\cdots, d_r$ respectively. We call the product $d_1d_2\cdots d_r$
the total degree, denoted by $d$. It is well known that all complete intersections of fixed multidegree are diffeomorphic. On the other hand, there exist diffeomorphic complete intersections with different multidegrees.
For lower dimensions, such as complex dimensions $2,3,4$, the diffeomorphic examples can be found in \cite{B,B00,LW2}. W. Ebeling (\cite{E}) and A.S. Libgober-J. Wood (\cite{LW5}) independently found examples of homeomorphic complex $2$-dimensional complete intersections but not diffeomorphic.
In \cite{FW}, F.Q. Fang and the author proved that, in dimensions $n=5,6,7$, two complete intersections $X_n(\underline{d})$ and $X_n(\underline{d}^\prime)$ are homeomorphic if and only if they
have the same total degree, Pontrjagin classes and Euler characteristics. Particularly, by Traving's result (\cite[Theorem A]{Kr} or \cite{Tr}), to the prime factorization of total degree $d=\prod_{p~\textrm{primes}}p^{\nu_{p}(d)}$, if
$\nu_p(d)\geqslant \frac{2n+1}{2(p-1)}+1$ for all primes $p$ with $p(p-1)\leqslant n+1$,
two homeomorphic complex $n$-dimensional complete intersections are diffeomorphic.

The first purpose of this paper is to give examples of diffeomorphic complex $5$-dimensional complete intersections with different multidegrees. These examples, which are easy to check but hard to happen upon, were found by computer search. From these examples, we can deduce some interesting remarks about complete intersections.

Libgober and Wood (\cite{LW4}) showed the existence of homeomorphic complete intersections of dimension $2$ and diffeomorphic ones of dimension $3$ which belong to components of the moduli space having different dimensions. In fact it was shown that there is a procedure which allows one to produce from a pair of homeomorphic complete intersections an arbitrarily long family, all members of which are homeomorphic. P. Br{\"u}ckmann (\cite{B}) shows that the construction mentioned yields families of arbitrary length $t$ of complete intersections in $\cp^{4t-2}$ (resp. $\cp^{5t-2}$) consisting of homeomorphic complete intersections of dimension $2$ (resp. diffeomorphic ones of dimension $3$) but that belong to components of the moduli space of different dimensions. Furthermore, under Theorem 1 of \cite{FK}, Br{\"u}ckmann also proves the similar result for the complete intersections of dimension $4$ in $\cp^{6t-2}$(\cite{B00}).

Another purpose of this paper is to give the following theorem, which is a supplement to the results of Br{\"u}ckmann \cite{B,B00}.
\begin{theorem}\label{main}
For each integer $t>1$, there exist $t$ diffeomorphic complex $5$-dimensional complete intersections in $\cp^{7t-2}$ isomorphism class of which lie in different dimensional components of the moduli space.
\end{theorem}

This paper is organized as follows: After presenting the basic formulas of characteristic classes of complete intersections in Section 2,  we will give examples of diffeomorphic complex $5$-dimensional complete intersections in Section 3. Section 4 proves Theorem \ref{main}. The last section will be devoted to the code of computer program to evaluate an inequality, which is a key to prove Theorem \ref{main}.

{\bf Acknowledgement.} This work was undertaken when the author visited the Department of Mathematical Sciences in University of Copenhagen. The author is grateful to Professor Jesper Michael M{\o}ller and Department of Mathematical Sciences for their hospitality. The author would like to thank the following students for their help on computer programming to search examples: Jianpeng Du, Mo Jia, Wenyu He, Sibo Zhao.

\section{Characteristic classes of complete intersections}

For a complete intersection
$X_n(\underline{d})$, let $H$ be the restriction of the dual bundle of the canonical line bundle over $\cp^{n+r}$ to $X_n(\underline{d})$, and $x= c_1(H)\in H^2(X_n(\underline{d});\zz)$.
Associate the multidegree $\underline{d}=(d_1,d_2,\dots,d_r)$, define the power sums $s_i=\sum_{j=1}^{r}d_j^i$ for $1\leqslant i\leqslant n$.
Then the Chern classes and Pontrjagin classes are presented as follows (\cite{LW2}):
\begin{align*}
c_k & =\frac{1}{k!}g_k(n+r+1-s_1,\dots,n+r+1-s_k)x^k, 1\leqslant k\leqslant n,\\
p_k & =\frac{1}{k!}g_k(n+r+1-s_2,\dots,n+r+1-s_{2k})x^{2k}, 1\leqslant k\leqslant
\left[\frac{n}{2}\right].
\end{align*}
The Euler characteristic is ($x^n\cap[X_n(\underline{d})]=d=d_1\cdots d_r$)
\begin{equation*}
e(X_n(\underline{d}))=c_n(X_n(\underline{d}))\cap [X_n(\underline{d})]
=d\frac{1}{n!}g_n(n+r+1-s_1,\dots,n+r+1-s_n).
\end{equation*}
Where the $g_k$'s are polynomials that can be iteratively computed from the Newton formula:
\begin{equation*}
s_k-g_1(s_1)s_{k-1}+\frac{1}{2}g_2(s_1,s_2)s_{k-2}+\cdots+(-1)^k\frac{1}{k!}g_k(s_1,s_2,\dots,s_k) k=0,k\geqslant 1.
\end{equation*}
For example, the first six are
\begin{align*}
g_1(s_1) & =s_1,\\
g_2(s_1,s_2) & =s_1^2-s_2,\\
g_3(s_1,s_2,s_3) & =s_1^3-3s_1s_2+2s_3,\\
g_4(s_1,\dots,s_4) & =s_1^4-6s_1^2s_2+8s_1s_3+3s_2^2-6s_4,\\
g_5(s_1,\dots,s_5) & =s_1^5-10s_1^3s_2+20s_1^2s_3-30s_1s_4+15s_1s_2^2-20s_2s_3+24s_5,\\
g_6(s_1,\dots,s_6) & =s_1^6-15s_1^4s_2+40s_1^3s_3-90s_1^2s_4+45s_1^2s_2^2-120s_1s_2s_3+144s_1s_5\\
& \quad -15s_2^3+90s_2s_4+40s_3^2-120s_6.
\end{align*}

Note that the $k^{\textrm{th}}$ Pontrjagin class $p_k$ is a integral multiple of $x^{2k}$, where $x$ generates
the second cohomology of the complete intersection. Thus we can compare this invariant for different
complete intersections. For convenience, throughout the rest of the paper,
we view the Pontrjagin class $p_k$ of $X_n(\underline{d})$ as the multiple of $x^{2k}$.

\section{Examples of diffeomorphic complex 5-dimensional complete intersections}

For complex $5$-dimensional complete intersections $X_5(d_1,\dots,d_r)$, its total degree, Pontrjagin classes and Euler characteristic are as follows:
 \begin{align}
 d & =d_1\times\cdots\times  d_r, \label{d} \\
 p_1 & =6+r-s_2, \label{p1} \\
 p_2 & =\frac{1}{2}{\big [}(6+r-s_2)^2-(6+r-s_4)\big{]}, \label{p2} \\
 e & =\frac{1}{5!}d\big{[}(6+r-s_1)^5-10(6+r-s_1)^3(6+r-s_2)+20(6+r-s_1)^2(6+r-s_3)\nonumber \\
 & \quad -30(6+r-s_1)(6+r-s_4)
 +15(6+r-s_1)(6+r-s_2)^2 \nonumber \\
  & \quad -20(6+r-s_2)(6+r-s_3)+24(6+r-s_5)\big ]. \label{e}
 \end{align}
Here, $p_1$ and $p_2$ denote the Pontrjagin classes as appointed in the end of Section 2.

By Theorem 1.1 of \cite{FW}, to find homeomorphic complex $5$-dimensional complete intersections, we only need to find different multidegrees, such that \eqref{d}-\eqref{e} all agree respectively. Additionally, by \cite[Theorem A]{Kr}, for the total degree $d=\prod_{p \textrm{~primes}}p^{\nu_p(d)}$, if $\nu_2(d)\geqslant 7$ and $\nu_3(d)\geqslant 4$, the homeomorphic $5$-dimensional complete intersections are diffeomorphic. This searching can completely be done by computer. According to \cite[Proposition 7.3]{LW2}, let $X_n(\underline{d})\subset \cp^{n+r}$ be a complete intersection of given codimension $r$ with $n>2$ and $2r\leqslant n+2$, then the total degree and Pontrjagin classes of $X_n(\underline{d})$ determine the multidegree.  Thus, it is impossible to find out such a homeomorphic or diffeomorphic example with different multidegrees in which one of the complete intersections has codimension $2$ or $3$ for complex dimension $5$. Theoretically, there should exist a lot of homeomorphic complete intersections with codimension $\geqslant 4$. However, with the codimension becoming smaller, it will become more difficult to find out such examples. In fact, we can offer such examples with codimension $7$ (See Section 4).

\begin{example}
Take two complete intersections $X_5 (46,36,34,21,14,13,12,11,3,2,2)$, and $X_5(44,42,26,23,18,17,7,6,6,4)$, we calculated the power sums of two multidegrees as follows:
\[
\begin{array}{c|ccccc} \hline
\textrm{Multidegree}  &  s_1 & s_2 & s_3 &  s_4 & s_5 \\ \hline
(46,36,34,21,14,13,12,11,3,2,2)   & 194 & 5656 & 200600 & 7790356 & 317267984 \\
(44,42,26,23,18,17,7,6,6,4)  & 193 & 5655 & 200599 & 7790355 & 317267983 \\  \hline
\end{array}
\]

\noindent Although, the above two complete intersections have different power sums and codimensions, they have the same total degree and symmetric functions $r-s_1,\dots,r-s_5$. By formulas \eqref{p1},\eqref{p2},\eqref{e}, it is evident that they have the same Pontrjagin classes and Euler characteristic.
\[
\begin{array}{c|ccccc} \hline
X_5(\underline{d})~(\textrm{codim=} 11,10) & d  & p_1 & p_2 & e/d \\ \hline
X_5(46,36,34,21,14,13,12,11,3,2,2) & 340867118592 & -5639 & 19794330 & -6401091783 \\
X_5(44,42,26,23,18,17,7,6,6,4) & 340867118592 & -5639 & 19794330 & -6401091783 \\ \hline
\end{array}
\]

\noindent Since total degree satisfies $d=2^9\times 3^5\times 7^2\times 11\times 13\times 17\times 23$,
they are diffeomorphic complete intersections.
\end{example}
\begin{example}\label{5t5f}
(1), Take $X_5 (66,56,45,39,16,15,8,3), X_5(64,60,42,39,20,11,9,3)$, it is easy to get the following table:
\begin{align*}
& \hskip .5cm \begin{array}{c|ccccc} \hline
\textrm{Multidegree}  & s_1 & s_2 &  s_3 & s_4 &  s_5 \\ \hline
(66,56,45,39,16,15,8,3) & 248 & 11592 & 621566 & 35343636 & 2079657638 \\
(64,60,42,39,20,11,9,3) & 248 & 11592 & 621638 & 35343636 & 2075677598 \\  \hline
\end{array}\\
&
\begin{array}{c|cccc} \hline
X_5(\underline{d})~(\textrm{codim=} 8) & d & p_1 & p_2 & e/d \\ \hline
X_5(66,56,45,39,16,15,8,3) & 37362124800 & -11578  & 84696853  & -31485015068 \\
X_5(64,60,42,39,20,11,9,3) & 37362124800 & -11578  & 84696853  & -31485015068 \\ \hline
\end{array}
\end{align*}

\noindent The above two multidegrees have different power sums $s_3, s_5$, but they have the same total degree, Pontrjagin classes and Euler characteristic.
Since $d=37362124800=2^{11}\times 3^6\times 5^2\times 7\times 11\times 13$, so $X_5(66,56,45,39,16,15,8,3)$ and $X_5(64,60,42,39,20,11,9,3)$
are diffeomorphic.

(2), By deleting the last degree $3$ from the multidegrees in (1), we take complete intersections $ X_5(66,56,45,39,16,15,8)$ and $X_5(64,60,42,39,20,11,9)$ .
\[
\begin{array}{c|cc} \hline
X_5(\underline{d})~(\textrm{codim=} 7) & d  & e/d\\ \hline
X_5(66,56,45,39,16,15,8) & 12454041600  & -30762573120 \\
X_5(64,60,42,39,20,11,9) & 12454041600  & -30762561840 \\  \hline
\end{array}
\]
The different Euler characteristics imply that $X_5 (66,56,45,39,16,15,8)$ is not homotopy equivalent to $X_5(64,60,42,39,20,11,9)$.

(3), By appending a degree $7$ into the multidegrees in (1), we find that complete intersections $X_5 (66,56,45,39,16,15,8,7,3)$ and $X_5(64,60,42,39,20,11,9,7,3)$ have different Euler characteristics,

\[
\begin{array}{c|cc} \hline
X_5(\underline{d})~(\textrm{codim=} 9) & d & e/d \\ \hline
X_5(66,56,45,39,16,15,8,7,3) & 261534873600 & -33795490160 \\
X_5(64,60,42,39,20,11,9,7,3) & 261534873600 & -33795524864 \\ \hline
\end{array}
\]

\noindent So $X_5 (66,56,45,39,16,15,8,7,3)$ and $X_5(64,60,42,39,20,11,9,7,3)$ are not homotopy equivalent.
\end{example}

\begin{example}\label{5t4f}
 For complex $4$-dimensional complete intersections $X_4(d_1,\dots,d_r)$, its Euler characteristic is as follows:
\begin{align*}
e & =\frac{d}{4!}\big{[}(5+r-s_1)^4-6(5+r-s_1)^2(5+r-s_2)+8(5+r-s_1)(5+r-s_3)\\
 & \quad +3(5+r-s_2)^2-6(5+r-s_4)\big ].
 \end{align*}
Let $X_5(66,56,45,39,16,15,8,3)$ and $X_5(64,60,42,39,20,11,9,3)$, which are diffeomorphic by Example \ref{5t5f} (1), simultaneously make transversal intersection with hypersurface of homogeneous degree $2$, we can construct two  complex $4$-dimensional complete intersections
$X_4(66,56,45,39,16,15,8,3,2)$ and $X_4(64,60,42,39,20,11,9,3,2)$. They have different Euler characteristics,
\[
\begin{array}{c|cc} \hline
X_4(\underline{d})~(\textrm{codim=} 9) & d & e/d \\ \hline
X_4(66,56,45,39,16,15,8,3,2) & 74724249600 &  365019422 \\
X_4(64,60,42,39,20,11,9,3,2) & 74724249600  &  365025086 \\ \hline
\end{array}
\]

\noindent So $X_4(66,56,45,39,16,15,8,3,2)$ and
$X_4(64,60,42,39,20,11,9,3,2)$ are not homotopy equivalent.
That is, although $X_5(66,56,45,39,16,15,8,3)$ and
$X_5(64,60,42,39,20,11,9,3)$ are diffeomorphic,
$X_4(66,56,45,39,16,15,8,3,2)$ and
$X_4(64,60,42,39,20,11,9,3,2)$, which are the transversal intersection of diffeomorphic complex $5$-dimensional complete intersections with the same hypersurface of homogeneous degree $2$, do not have the same homotopy type.
\end{example}

\begin{example}\label{6f5t}
 For complex $6$-dim complete intersections $X_6(d_1,\dots,d_r)$, its Euler characteristic is as follows:
\begin{align*}
 e & =\frac{d}{6!}\big{[}(7+r-s_1)^6-15(7+r-s_1)^4(7+r-s_2)+40(7+r-s_1)^3(7+r-s_3)\\
 & \quad -90(7+r-s_1)^2(7+r-s_4)+45(7+r-s_1)^2(7+r-s_2)^2\\
 & \quad -120(7+r-s_1)(7+r-s_2)(7+r-s_3)+144(7+r-s_1)(7+r-s_5)\\
 & \quad -15(7+r-s_2)^3+90(7+r-s_2)(7+r-s_4)+40(7+r-s_3)^2-120(7+r-s_6)\big ].
\end{align*}
Take $X_6(66,56,45,16,15,8,3), X_6(64,60,42,20,11,9,3)$, it is easy to check that
\[
\begin{array}{c|cc} \hline
X_6(\underline{d})~(\textrm{codim=} 7) & d  & e/d \\ \hline
X_6(66,56,45,16,15,8,3) & 958003200  &   1370218430570\\
X_6(64,60,42,20,11,9,3) & 958003200  &   1369971514442\\ \hline
\end{array}
\]

\noindent The different Euler characteristics imply that $X_6(66,56,45,16,15,8,3)$ is not homotopy equivalent to $X_6(64,60,42,20,11,9,3)$. However,
$X_5(66,56,45,39,16,15,8,3)$ and \\ $X_5(64,60,42,39,20,11,9,3)$, which are the transversal intersection of the above two non-homotopy equivalent complex $6$-dimensional complete intersections with the same hypersurface of homogeneous degree $39$, are diffeomorphic by Example \ref{5t5f} (1).
\end{example}

Compare the above Examples \ref{5t5f}, \ref{5t4f} and \ref{6f5t}, we can obtain the following interesting remarks.

\begin{remark}\label{negativetoFang}
 $X_{n}(d_1,\dots,d_{r-1},c)$ is homeomorphic (diffeomorphic, homotopy equivalent) to $X_{n}(d_1^{\prime},\dots,d_{r-1}^{\prime},c)$, however, it may not be true not only for $X_{n}(d_1,\dots,d_{r-1})$ and $X_{n}(d_1^{\prime},\dots,d_{r-1}^{\prime})$,  but also for $X_{n}(d_1,\dots,d_{r-1},c,c^\prime)$ and $X_{n}(d_1^{\prime},\dots,d_{r-1}^{\prime},c,c^\prime)$(See Example \ref{5t5f} (1),(2),(3)).
\end{remark}
Note that, in \cite{F}, Fang asked the following question: If $X_n(\underline{d})$ and $X_n(\underline{d}^\prime)$ are diffeomorphic/or homeomorphic/or homotopy equivalent, is $X_n(\underline{d};a)$ diffeomorphic to $X_n(\underline{d}^\prime;a)$ for a natural number $a$? Here $X_n(\underline{d};a)$ is the complete intersection with multidegree $(d_1,d_2,\dots,d_r,a)$. Now, Remark \ref{negativetoFang} partially gives a negative answer to Fang's question.

\begin{remark}
 $X_{n+1}(d_1,\dots,d_{r-1})$ is diffeomorphic to $X_{n+1}(d_1^{\prime},\dots,d_{r-1}^{\prime})$,  but it may not be true for $X_{n}(d_1,\dots,d_{r-1},c)$ and $X_{n}(d_1^{\prime},\dots,d_{r-1}^{\prime},c)$ (See Example \ref{5t4f}), even if $c\leqslant \min\{\underline{d},\underline{d}^\prime\}$.
\end{remark}

\begin{remark}
Even if $X_{n+1}(d_1,\dots,d_{r-1})$ is not diffeomorphic to $X_{n+1}(d_1^{\prime},\dots,d_{r-1}^{\prime})$,  \\ $X_{n}(d_1,\dots,d_{r-1},c)$ can be diffeomorphic to $X_{n}(d_1^{\prime},\dots,d_{r-1}^{\prime},c)$ (See Example \ref{6f5t}).
\end{remark}

\section{Moduli spaces of complete intersections}

In this section, we will prove Theorem \ref{main}.

Let $X_n(\underline{d})\subset \cp^N$, where $n\geqslant 2, \underline{d}=(d_1,\dots,d_r), d_i\geqslant 2$ and $r=N-n$. Then from \cite[Lemma 3]{B}, the explicit formula for moduli space dimension is
\begin{align}\label{dimM}
  m(\underline{d})\triangleq & ~m(X_n(\underline{d}))= 1-(N+1)^2+\sum_{i=1}^r\binom{N+d_i}{N} \nonumber \\
& +\sum_{i=1}^r\sum_{j=1}^r(-1)^j\sum_{1\leqslant k_1<\cdots<k_j\leqslant r}\binom{N+d_i-d_{k_1}-\cdots-d_{k_{j}}}{N}.
\end{align}
Where $\displaystyle\binom{m}{N}=0$ for $m<N(m\in \zz)$.

\begin{theorem}
For each integer $t>1$, there exist $t$ diffeomorphic complex $5$-dimensional complete intersections in $\cp^{7t-2}$ isomorphism class of which lie in different dimensional components of the moduli space.
\end{theorem}

\begin{proof} 
Consider the following two multidegrees
\begin{center}
$\underline{d}=(88,77,72,54,48,31,29),~ \underline{d}^\prime=(87,81,64,62,44,33,28)$.
\end{center}
We list the corresponding power sums, total degree, Pontrjagin classes and Euler characteristic in Table \ref{codim7}.
\begin{table}[h]
\begin{align*}
&\hskip .8cm \begin{array}{c|ccccc} \hline
 & s_1 & s_2 & s_3 & s_4 & s_5 \\ \hline
\underline{d} & 399 & 25879 & 1833489 & 137438707 & 10682130249 \\
\underline{d}^\prime & 399 & 25879 & 1833489 & 137438707 & 10682130249 \\ \hline
\end{array} \\
&
\begin{array}{c|*{5}c} \hline
 & d & p_1 & p_2 & e/d \\ \hline
X_5(\underline{d}) & 1136843237376 & -25866 & 403244325 & -296492615140\\
X_5(\underline{d}^\prime) & 1136843237376 & -25866 & 403244325 & -296492615140\\ \hline
\end{array}
\end{align*}
\caption{Power sum, total degree, Pontrjagin class, Euler characteristic}
\label{codim7}
\end{table}

From Table \ref{codim7}, the total degree is $1136843237376=2^{11}\times 3^6\times 7\times 11^2\times 29\times 31$, so the two complete intersections $X_5(\underline{d})$ and $X_5(\underline{d}^\prime)$ are diffeomorphic but have different moduli space dimensions:
\begin{align*}
m(\underline{d}) & = \fnum{1382270197857128},\\
m(\underline{d}^\prime) & = \fnum{1370693416581393}.
\end{align*}
There is a way to generate larger sets of diffeomorphic complete intersections from the above pairs $\underline{d}$  and $\underline{d}^\prime$, which arose from \cite{LW4} and had an application in \cite{B,B00}.

Denote the composed multidegree
\begin{center}
$d_{\lambda,\mu}=(\underbrace{\underline{d},\dots,\underline{d}}_{\lambda},
\underbrace{\underline{d}^\prime,\dots,\underline{d}^\prime}_{\mu}), \lambda+\mu=s\geqslant 1$ .
\end{center}
Then  the composed multidegrees $d_{0,s},d_{1,s-1},\dots,d_{s,0}$ have the same power sums $s_1,s_2,\dots,s_5$ respectively, so the corresponding complete intersections are diffeomorphic to each other.
Let $X_5(d_{\lambda,\mu})\subset \cp^{7s+5}$ be $5$-dimensional complete intersections with multidegree $d_{\lambda,\mu}$.
It is reasonable to expect that the corresponding $m(d_{\lambda,\mu})$'s will all be different(See \cite{LW4}). There is no general way to prove this. However, for the dimension formula \eqref{dimM} and the above special pairs $\underline{d}$ and $\underline{d}^\prime$, there are finite binomial coefficients $\displaystyle\binom{N+d_i-d_{k_1}-\cdots-d_{k_{j}}}{N}$ different from zero $(N=7s+5)$. We can prove the following inequality:
\begin{equation*}
 m(d_{\lambda+1,\mu-1})-m(d_{\lambda,\mu})>0,  0\leqslant \lambda<s=\lambda+\mu.
\end{equation*}
This inequality will be proved in the coming Proposition.

Now, the sequence $m(d_{\lambda,s-\lambda})|_{\lambda=0,1,\dots,s-1}$ is strictly monotonously increasing. Let $t=s+1$, there exist $t$ five-dimensional complete intersections $X_5(d_{0,s}), X_5(d_{1,s-1}), \dots, X_5(d_{s,0})$ in $\cp^{7s+5}=\cp^{7t-2}$ with the desired properties. The proof is finished.
\end{proof}
\begin{proposition}\label{Pro-mis}
\begin{equation*}
 m(d_{\lambda+1,s-\lambda-1})-m(d_{\lambda,s-\lambda})>0,  0\leqslant \lambda<s.
\end{equation*}
\end{proposition}
\begin{proof}
For the chosen multidegrees $\underline{d}$ and $\underline{d}^\prime$,
\begin{align*}
 m(d_{\lambda,s-\lambda}) = & 1-(N+1)^2+\Big[\lambda\sum_{d_i\in\underline{d}}+
(s-\lambda)\sum_{d_i\in\underline{d}^\prime}\Big]\binom{N+d_i}{N}\\
& +\Big[\lambda\sum_{d_i\in\underline{d}}+(s-\lambda)\sum_{d_i\in\underline{d}^\prime}\Big]
\sum_{j=1}^3(-1)^j\hskip -.3cm\sum_{1\leqslant k_1<\cdots<k_j\leqslant 7s \atop d_{k_1},\dots,d_{k_j}\in d_{\lambda,s-\lambda}}\hskip -.3cm\binom{N+d_i-d_{k_1}-\cdots-d_{k_{j}}}{N},
\end{align*}
Where, the index $j$ is maximally $3$ that is determined by $\max\{\underline{d},\underline{d}^\prime\}=88$ and $\min\{\underline{d},\underline{d}^\prime\}=28$.  So,
\begin{align}\label{DiffdimM}
&  m(d_{\lambda+1,s-\lambda-1})- m(d_{\lambda,s-\lambda})  \nonumber\\
 = &\Big[\sum_{d_i\in\underline{d}}-\sum_{d_i\in\underline{d}^\prime}\Big]\binom{N+d_i}{N} \nonumber\\
& +\Big[(\lambda+1)\sum_{d_i\in\underline{d}}+(s-\lambda-1)\sum_{d_i\in\underline{d}^\prime}
\Big]\sum_{j=1}^3(-1)^j\hskip -.5cm
\sum_{1\leqslant k_1<\cdots<k_j\leqslant 7s  \atop d_{k_1},\dots,d_{k_j}\in d_{\lambda+1,s-\lambda-1}}
\hskip -.2cm \binom{N+d_i-d_{k_1}-\cdots-d_{k_{j}}}{N}  \nonumber\\
& -\Big[\lambda\sum_{d_i\in\underline{d}}+(s-\lambda)\sum_{d_i\in\underline{d}^\prime}\Big]
\sum_{j=1}^3(-1)^j\hskip -.5cm
\sum_{1\leqslant k_1<\cdots<k_j\leqslant 7s \atop d_{k_1},\dots,d_{k_j}\in d_{\lambda,s-\lambda}}\hskip -.2cm \binom{N+d_i-d_{k_1}-\cdots-d_{k_{j}}}{N} \\
\triangleq & \sum_{j=0}^3 M_{j}(\lambda,s) , \nonumber
\end{align}
To prove \eqref{DiffdimM} $>0$, let us decompose  \eqref{DiffdimM} into the sum of $M_{j}(\lambda,s), j=0,1,2,3$. In the following, we will describe $M_{j}(\lambda,s)$ as polynomials of invariants $s$ and $\lambda~ (N=7s+5)$.  Firstly,
\begin{align}\label{j=0}
M_{0}(\lambda,s)\triangleq & \Big[\sum_{d_i\in\underline{d}}-\sum_{d_i\in\underline{d}^\prime}\Big]\binom{N+d_i}{N},\\
M_{1}(\lambda,s)\triangleq &\Big[
-(\lambda+1)\sum_{d_i\in\underline{d}}\sum_{d_{k}\in d_{\lambda+1,s-\lambda-1}}-(s-\lambda-1)\sum_{d_i\in\underline{d}^\prime}\sum_{d_{k}\in d_{\lambda+1,s-\lambda-1}} \nonumber\\
& +\lambda\sum_{d_i\in\underline{d}}\sum_{d_{k}\in d_{\lambda,s-\lambda}}+(s-\lambda)\sum_{d_i\in\underline{d}^\prime}\sum_{d_{k}\in d_{\lambda,s-\lambda}}
\Big]\binom{N+d_i-d_{k}}{N} \nonumber\\
= &\Big[
-(\lambda+1)^2\sum_{d_i\in\underline{d}}\sum_{d_{k}\in \underline{d}}-(\lambda+1)(s-\lambda-1)\sum_{d_i\in\underline{d}}\sum_{d_{k}\in \underline{d}^\prime}  \nonumber\\
& -(s-\lambda-1)(\lambda+1)\sum_{d_i\in\underline{d}^\prime}\sum_{d_{k}\in \underline{d}}-(s-\lambda-1)^2\sum_{d_i\in\underline{d}^\prime}\sum_{d_{k}\in \underline{d}^\prime} \nonumber\\
& +\lambda^2\sum_{d_i\in\underline{d}}\sum_{d_{k}\in \underline{d}}+\lambda(s-\lambda)\sum_{d_i\in\underline{d}}\sum_{d_{k}\in \underline{d}^\prime}  \nonumber \\
& +(s-\lambda)\lambda\sum_{d_i\in\underline{d}^\prime}\sum_{d_{k}\in \underline{d}}+(s-\lambda)^2\sum_{d_i\in\underline{d}^\prime}\sum_{d_{k}\in \underline{d}^\prime}
\Big]\binom{N+d_i-d_{k}}{N} \nonumber\\
= &\Big[
(-2\lambda-1)\sum_{d_i\in\underline{d}}\sum_{d_{k}\in \underline{d}}+(1+2\lambda-s)\sum_{d_i\in\underline{d}}\sum_{d_{k}\in \underline{d}^\prime}  \nonumber\\
& +(1+2\lambda-s)\sum_{d_i\in\underline{d}^\prime}\sum_{d_{k}\in \underline{d}}+(2s-2\lambda-1)\sum_{d_i\in\underline{d}^\prime}\sum_{d_{k}\in \underline{d}^\prime}
\Big]\binom{N+d_i-d_{k}}{N}. \label{j=1}
\end{align}
There are four summations in the third part $M_{2}(\lambda,s)$,
\begin{align*}
M_{2}(\lambda,s)\triangleq &\Big[
(\lambda+1)\sum_{d_i\in\underline{d}}\sum_{1\leqslant k_1<k_2\leqslant 7s \atop d_{k_1},d_{k_2}\in d_{\lambda+1,s-\lambda-1}}+(s-\lambda-1)\sum_{d_i\in\underline{d}^\prime}\sum_{1\leqslant k_1<k_2\leqslant 7s \atop d_{k_1},d_{k_2}\in d_{\lambda+1,s-\lambda-1}}\\
& -\lambda\sum_{d_i\in\underline{d}}\sum_{1\leqslant k_1<k_2\leqslant 7s \atop d_{k_1},d_{k_2}\in d_{\lambda,s-\lambda}}-(s-\lambda)\sum_{d_i\in\underline{d}^\prime}\sum_{1\leqslant k_1<k_2\leqslant 7s \atop d_{k_1},d_{k_2}\in d_{\lambda,s-\lambda}}
\Big]\binom{N+d_i-d_{k_1}-d_{k_2}}{N},
\end{align*}
For the simplification of summations, let's define
\begin{align*}
\Gamma_{\underline{d}\underline{d}^\prime\underline{d}} &
=\sum_{d_i\in \underline{d}}\sum_{d_j\in \underline{d}^\prime}\sum_{d_k\in \underline{d}}\binom{N+d_i-d_j-d_k}{N}=\Gamma_{\underline{d}\underline{d}\underline{d}^\prime},\\
\Gamma_{\underline{d}^\prime\underline{d}^\prime\underline{d}} &
=\sum_{d_i\in \underline{d}^\prime}\sum_{d_j\in \underline{d}^\prime}\sum_{d_k\in \underline{d}}\binom{N+d_i-d_j-d_k}{N} =\Gamma_{\underline{d}^\prime\underline{d}\underline{d}^\prime} .\\
\Gamma_{\underline{d}\underline{d}_<} & =\sum_{d_i\in \underline{d}}\sum_{1\leqslant k_1<k_2\leqslant 7 \atop d_{k_1},d_{k_2}\in \underline{d}}\binom{N+d_i-d_{k_1}-d_{k_2}}{N}, \\
\Gamma_{\underline{d}^\prime\underline{d}_<} & =\sum_{d_i\in \underline{d}^\prime}\sum_{1\leqslant k_1<k_2\leqslant 7 \atop d_{k_1},d_{k_2}\in \underline{d}}\binom{N+d_i-d_{k_1}-d_{k_2}}{N}.
\end{align*}
Similarly, $\Gamma_{\underline{d}\underline{d}\underline{d}}, \Gamma_{\underline{d}^\prime\underline{d}^\prime\underline{d}^\prime},
\Gamma_{\underline{d}\underline{d}^\prime\underline{d}^\prime}, \Gamma_{\underline{d}^\prime\underline{d}\underline{d}},
\Gamma_{\underline{d}\underline{d}^\prime_<}, \Gamma_{\underline{d}^\prime\underline{d}^\prime_<}$ can also be imitated and defined.
By induction, it is easy to see that:
\begin{align*}
& \sum_{d_i\in \underline{d}}\sum_{1\leqslant k_1<k_2\leqslant 7s \atop d_{k_1},d_{k_2}\in d_{\lambda,s-\lambda}}\binom{N+d_i-d_{k_1}-d_{k_2}}{N}\\
= & \lambda\Gamma_{\underline{d}\underline{d}_<}+\frac{\lambda(\lambda-1)}{2}
\Gamma_{\underline{d}\underline{d}\underline{d}}+
\lambda(s-\lambda)\Gamma_{\underline{d}\underline{d}\underline{d}^\prime}+
(s-\lambda)\Gamma_{\underline{d}\underline{d}^\prime_<}+\frac{(s-\lambda)(s-\lambda-1)}{2}
\Gamma_{\underline{d}\underline{d}^\prime\underline{d}^\prime}.
\end{align*}
Similarly,
\begin{align*}
& \sum_{d_i\in \underline{d}^\prime}\sum_{1\leqslant k_1<k_2\leqslant 7s \atop d_{k_1},d_{k_2}\in d_{\lambda,s-\lambda}}\binom{N+d_i-d_{k_1}-d_{k_2}}{N}\\
= & \lambda\Gamma_{\underline{d}^\prime\underline{d}_<}+\frac{\lambda(\lambda-1)}{2}
\Gamma_{\underline{d}^\prime\underline{d}\underline{d}}+
\lambda(s-\lambda)\Gamma_{\underline{d}^\prime\underline{d}\underline{d}^\prime}+
(s-\lambda)\Gamma_{\underline{d}^\prime\underline{d}^\prime_<}+\frac{(s-\lambda)(s-\lambda-1)}{2}
\Gamma_{\underline{d}^\prime\underline{d}^\prime\underline{d}^\prime}.
\end{align*}
Then,
\begin{align}\label{j=2}
M_{2}(\lambda,s)= & (\lambda+1)\Big[(\lambda+1)\Gamma_{\underline{d}\underline{d}_<}+\frac{(\lambda+1)\lambda}{2}
\Gamma_{\underline{d}\underline{d}\underline{d}}+(\lambda+1)(s-\lambda-1)
\Gamma_{\underline{d}\underline{d}\underline{d}^\prime} \nonumber\\
& \hskip 1.cm+(s-\lambda-1)\Gamma_{\underline{d}\underline{d}^\prime_<} +\frac{(s-\lambda-1)(s-\lambda-2)}{2}
\Gamma_{\underline{d}\underline{d}^\prime\underline{d}^\prime}\Big]   \nonumber\\
 & +(s-\lambda-1)\Big[(\lambda+1)\Gamma_{\underline{d}^\prime\underline{d}_<}+
\frac{(\lambda+1)\lambda}{2} \Gamma_{\underline{d}^\prime\underline{d}\underline{d}}+(\lambda+1)(s-\lambda-1)
\Gamma_{\underline{d}^\prime\underline{d}\underline{d}^\prime} \nonumber\\
& \hskip 2.cm +(s-\lambda-1)\Gamma_{\underline{d}^\prime\underline{d}^\prime_<}+\frac{(s-\lambda-1)(s-\lambda-2)}{2}
\Gamma_{\underline{d}^\prime\underline{d}^\prime\underline{d}^\prime}\Big]   \nonumber\\
& -\lambda\Big[\lambda\Gamma_{\underline{d}\underline{d}_<}+\frac{\lambda(\lambda-1)}{2}
\Gamma_{\underline{d}\underline{d}\underline{d}}+\lambda(s-\lambda)
\Gamma_{\underline{d}\underline{d}\underline{d}^\prime}
+(s-\lambda)\Gamma_{\underline{d}\underline{d}^\prime_<}+\frac{(s-\lambda)(s-\lambda-1)}{2}
\Gamma_{\underline{d}\underline{d}^\prime\underline{d}^\prime}\Big]   \nonumber\\
& -(s-\lambda)\Big[\lambda\Gamma_{\underline{d}^\prime\underline{d}_<}+
\frac{\lambda(\lambda-1)}{2} \Gamma_{\underline{d}^\prime\underline{d}\underline{d}}+\lambda(s-\lambda)
\Gamma_{\underline{d}^\prime\underline{d}\underline{d}^\prime}
+(s-\lambda)\Gamma_{\underline{d}^\prime\underline{d}^\prime_<} \nonumber\\
& \hskip 3cm +\frac{(s-\lambda)(s-\lambda-1)}{2} \Gamma_{\underline{d}^\prime\underline{d}^\prime\underline{d}^\prime}\Big]   \nonumber\\
= & (2\lambda+1)\Gamma_{\underline{d}\underline{d}_<}+\frac{\lambda(3\lambda+1)}{2}
\Gamma_{\underline{d}\underline{d}\underline{d}}+\big[(\lambda+1)^2(s-\lambda-1)-
\lambda^2(s-\lambda)\big]
\Gamma_{\underline{d}\underline{d}\underline{d}^\prime}  \nonumber\\
& +(s-2\lambda-1)\big(\Gamma_{\underline{d}\underline{d}^\prime_<}+
\Gamma_{\underline{d}^\prime\underline{d}_<}\big)+\frac{(s-\lambda-1)(s-3\lambda-2)}{2}
\Gamma_{\underline{d}\underline{d}^\prime\underline{d}^\prime}  \nonumber\\
& +\frac{\lambda(2s-3\lambda-1)}{2}\Gamma_{\underline{d}^\prime\underline{d}\underline{d}}+
\big[(\lambda+1)(s-\lambda-1)^2-\lambda(s-\lambda)^2\big]
\Gamma_{\underline{d}^\prime\underline{d}\underline{d}^\prime} \nonumber\\
& +(1-2s+2\lambda)
\Gamma_{\underline{d}^\prime\underline{d}^\prime_<}+\frac{(s-\lambda-1)(2-3s+3\lambda)}{2} \Gamma_{\underline{d}^\prime\underline{d}^\prime\underline{d}^\prime}.
\end{align}
For the last part $M_{3}(\lambda,s)$,
\begin{align*}
& M_{3}(\lambda,s)\triangleq\Big[
-(\lambda+1)\sum_{d_i\in\underline{d}}\sum_{1\leqslant k_1<k_2<k_3\leqslant 7s \atop d_{k_1},d_{k_2},d_{k_3}\in d_{\lambda+1,s-\lambda-1}}-~(s-\lambda-1)\sum_{d_i\in\underline{d}^\prime}
\sum_{1\leqslant k_1<k_2<k_3\leqslant 7s \atop d_{k_1},d_{k_2},d_{k_3}\in d_{\lambda+1,s-\lambda-1}}\\
& \hskip 1cm +\lambda\sum_{d_i\in\underline{d}}
\sum_{1\leqslant k_1<k_2<k_3\leqslant 7s \atop d_{k_1},d_{k_2},d_{k_3}\in d_{\lambda,s-\lambda}}+~(s-\lambda)\sum_{d_i\in\underline{d}^\prime}
\sum_{1\leqslant k_1<k_2<k_3\leqslant 7s \atop d_{k_1},d_{k_2},d_{k_3}\in d_{\lambda,s-\lambda}}
\Big]\binom{N+d_i-d_{k_1}-d_{k_2}-d_{k_3}}{N}.
\end{align*}
According to $\underline{d}=(88,77,72,54,48,31,29), \underline{d}^\prime=(87,81,64,62,44,33,28)$, to make sure
$\displaystyle\binom{N+d_i-d_{k_1}-d_{k_2}-d_{k_3}}{N}$ nontrivial, $d_i$ can only be chosen from $88$ or $87$, and $d_{k_1}, d_{k_2}, d_{k_3}$ are chosen from $31, 29$ or $28$.  So
\begin{align*}
& M_{3}(\lambda,s) =\Big[
-(\lambda+1)\hskip -1cm\sum_{1\leqslant k_1<k_2<k_3\leqslant 7s \atop d_{k_1},d_{k_2},d_{k_3}\in d_{\lambda+1,s-\lambda-1}}+~\lambda\hskip -.5cm
\sum_{1\leqslant k_1<k_2<k_3\leqslant 7s \atop d_{k_1},d_{k_2},d_{k_3}\in d_{\lambda,s-\lambda}}\Big]\binom{N+88-d_{k_1}-d_{k_2}-d_{k_3}}{N} \\
& +\Big[-(s-\lambda-1)\hskip -1cm\sum_{1\leqslant k_1<k_2<k_3\leqslant 7s \atop d_{k_1},d_{k_2},d_{k_3}\in d_{\lambda+1,s-\lambda-1}}\hskip -.5cm
+~(s-\lambda)\hskip -.5cm\sum_{1\leqslant k_1<k_2<k_3\leqslant 7s \atop d_{k_1},d_{k_2},d_{k_3}\in d_{\lambda,s-\lambda}}
\Big]\binom{N+87-d_{k_1}-d_{k_2}-d_{k_3}}{N}.
\end{align*}
By induction, it is easy to see that
\begin{align*}
& \sum_{1\leqslant k_1<k_2<k_3\leqslant 7s \atop d_{k_1},d_{k_2},d_{k_3}\in d_{\lambda,s-\lambda}}\hskip -.5cm \binom{N+88-d_{k_1}-d_{k_2}-d_{k_3}}{N} \\
& = \lambda^2(s-\lambda)+\big[\lambda\binom{s-\lambda}{2}+\frac{1}{6}(\lambda-2)(\lambda-1)\lambda\big]
\binom{N+1}{N}+\frac{\lambda(\lambda-1)(s-\lambda)}{2}\binom{N+2}{N} \\
& \hskip .5cm +\lambda\binom{s-\lambda}{2} \binom{N+3}{N}+\frac{1}{6}(s-\lambda-2)(s-\lambda-1)(s-\lambda)\binom{N+4}{N},\\
& \sum_{1\leqslant k_1<k_2<k_3\leqslant 7s \atop d_{k_1},d_{k_2},d_{k_3}\in d_{\lambda,s-\lambda}}\hskip -.5cm \binom{N+87-d_{k_1}-d_{k_2}-d_{k_3}}{N} \\
& =\lambda\binom{s-\lambda}{2}+\frac{1}{6}(\lambda-2)(\lambda-1)\lambda +\frac{\lambda(\lambda-1)(s-\lambda)}{2}\binom{N+1}{N}  \\
& \hskip .5cm +\lambda\binom{s-\lambda}{2}
\binom{N+2}{N}+\frac{1}{6}(s-\lambda-2)(s-\lambda-1)(s-\lambda)\binom{N+3}{N}.
\end{align*}
Note that $M_{3}(\lambda,s)$ will non-trivially appear only when $s\geqslant 2$. Thus
\begin{align}\label{j=3}
M_{3}(\lambda,s)& = \frac{1}{6}(12-21s+12s^2-3s^3+44\lambda-54s\lambda+18s^2\lambda+60\lambda^2-48s\lambda^2+40\lambda^3)
\nonumber \\
& +\frac{1}{6}(-6+9s-3s^2-23\lambda+30s\lambda-12s^2\lambda-33\lambda^2+36s\lambda^2-28\lambda^3)
\binom{N+1}{N} \nonumber \\
& -\frac{1}{2}(-1+s-2\lambda)(2-3s+s^2+4\lambda-4s\lambda+4\lambda^2)\binom{N+2}{N} \nonumber \\
& +\frac{1}{3}(-1+s-\lambda)(6-7s+2s^2+13\lambda-7s\lambda+8\lambda^2)\binom{N+3}{N} \nonumber \\
& +\frac{1}{6}(1-s+\lambda)(2-s+\lambda)(3-s+4\lambda)\binom{N+4}{N}.
\end{align}

Summarize \eqref{j=0}-\eqref{j=3}, we see that \eqref{DiffdimM} is exactly a polynomial of $s, \lambda$ with complicated coefficients and higher degree. Fortunately, using the technical computational software  {\bf\emph{Mathematica}}, \eqref{j=0}-\eqref{j=3} can all be computed by executable program.
Finally, we calculate the following results:
\begin{align*}
& m(d_{1,0})-m(d_{0,1})=\fnum{11576781275735},\\
& m(d_{2,0})-m(d_{1,1})=\fnum{34356628415559239284},\\
& m(d_{1,1})-m(d_{0,2})=\fnum{34347842980758828832}.
\end{align*}
More generally,
\begin{equation*}
 m(d_{\lambda+1,s-\lambda-1})-m(d_{\lambda,s-\lambda})>
\left\{
\begin{array}{rl}
3148, & 0\leqslant \lambda <s,\\
4\times 10^{24}, & 0\leqslant \lambda <s, s\geqslant 3.
\end{array}
\right.
\end{equation*}
Furthermore, the outputs of the following two cases in {\bf\emph{Mathematica}} program are false,
\begin{equation*}
 m(d_{\lambda+1,s-\lambda-1})-m(d_{\lambda,s-\lambda})<
\left\{
\begin{array}{rl}
3148, & 0\leqslant \lambda <s,\\
4\times 10^{24}, & 0\leqslant \lambda <s, s\geqslant 3.
\end{array}
\right.
\end{equation*}
Thus, it is clear that, with any fixed $s\geqslant 1, s>\lambda\geqslant 0$, $m(d_{\lambda,s-\lambda})$ form a strictly monotonously increasing sequence for $\lambda$. Hence, the Proposition follows.
\end{proof}

\section{Mathematica Code and outputs}

In this section, {\bf\emph{Mathematica}} code and outputs that are designed to evaluate the inequality in Proposition \ref{Pro-mis} are attached in a notebook(.{\it nb} format).

\newpage

\includegraphics[trim = 25mm 0mm 0mm 38mm,scale=.9]{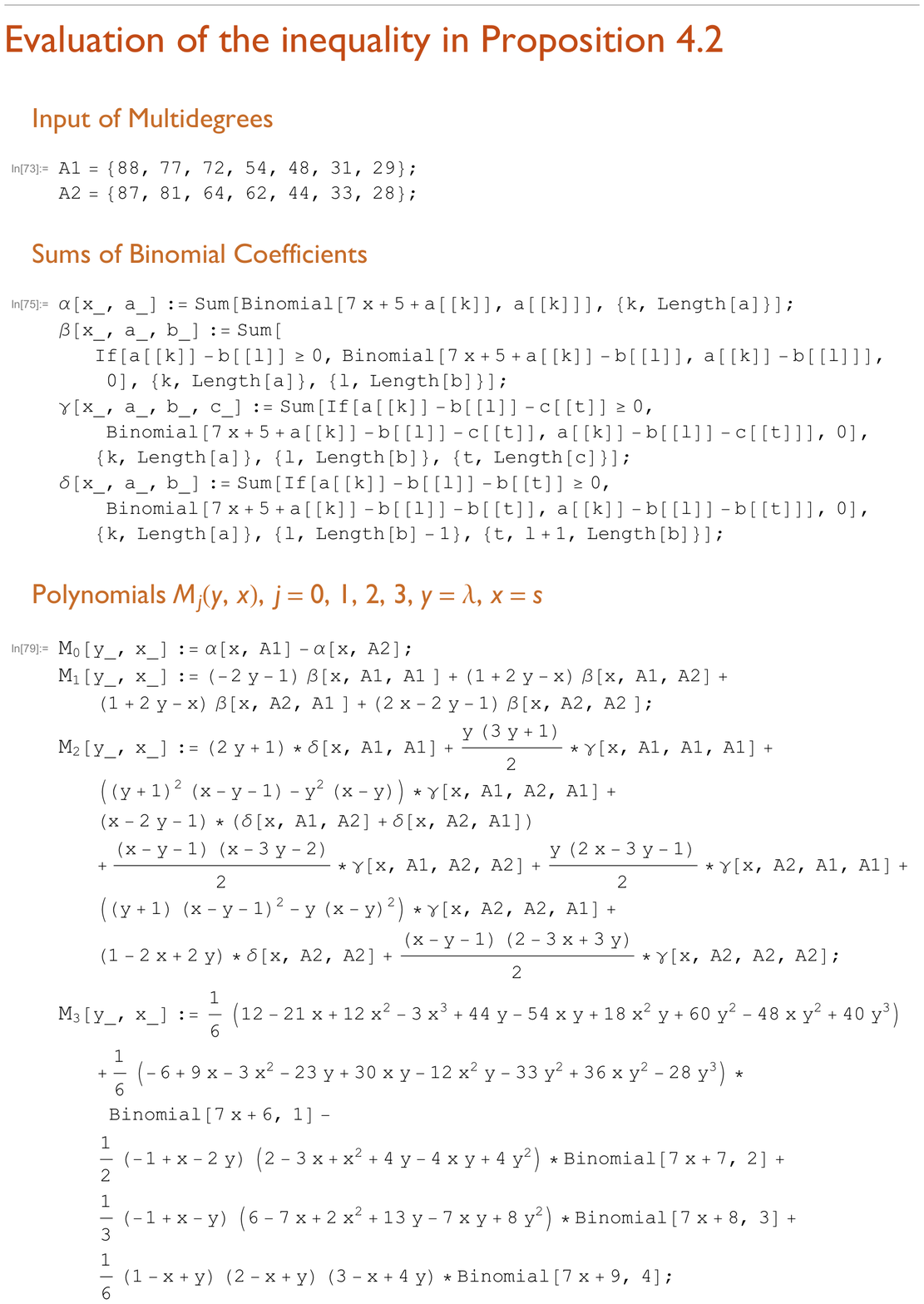}

\includegraphics[trim = 25mm 0mm 0mm 35mm,scale=.9]{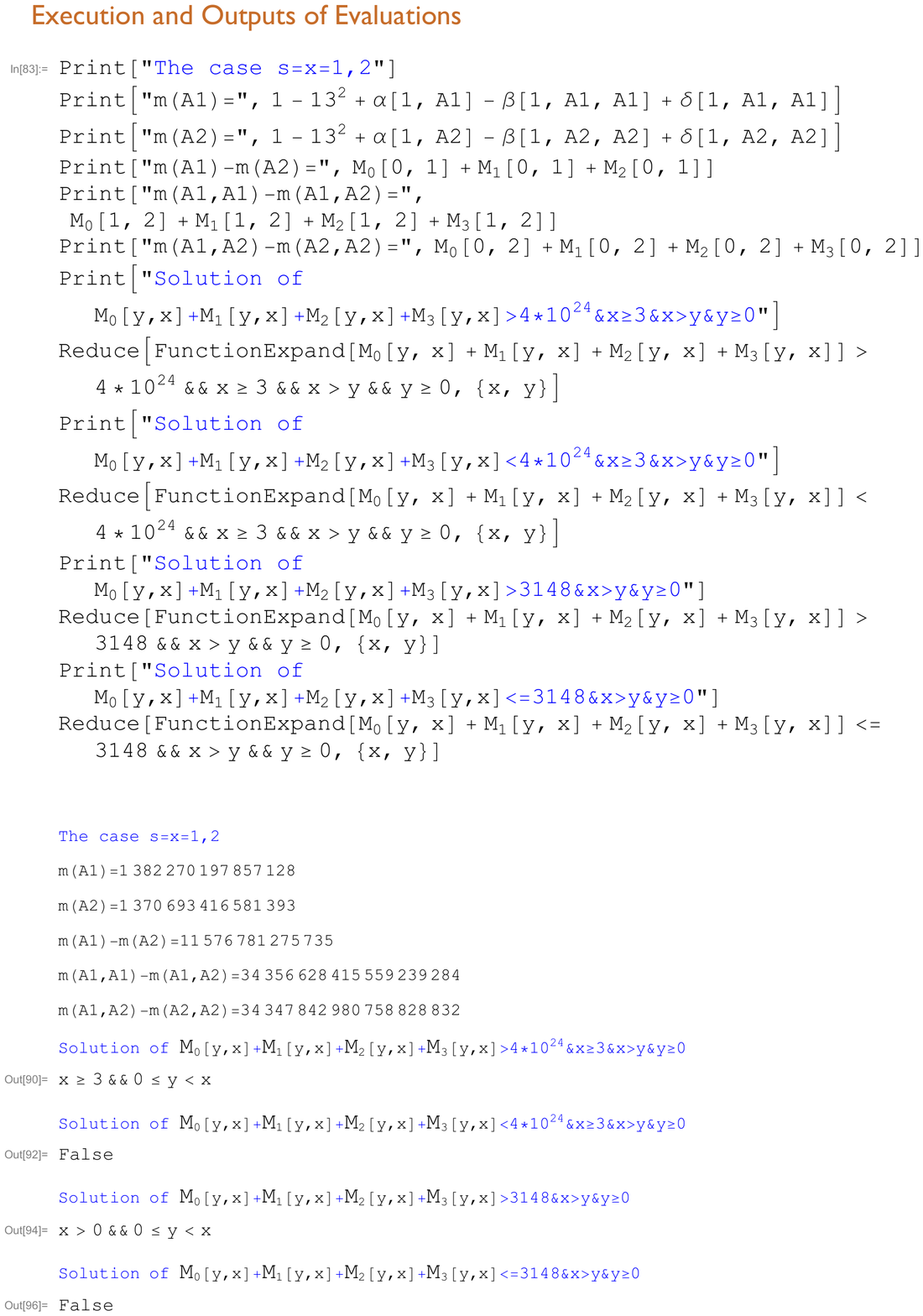}

\end{document}